\documentclass{amsart}

\usepackage{extramath,extratheorem,amscd}

\nc\arith{\mathrm{arith}}
\nc\mero{\mathrm{mero}}
\nc\dr{\mathrm{dR}}
\nc\sheafom{\underline{\boldsymbol\omega}}
\nc\ocan{\omega_{\mathrm{can}}}
\nc\xcan{\xi_{\mathrm{can}}}
\nc\wk{\mathrm{wk}}
\nc\wx{\mathrm{wk-ex}}
\nc\locz{\mathrm{loc}_z}
\dmo{\Sym}{Sym}
\dmo{\Res}{Res}
\dmo{\ord}{ord}
\nc\cE{\calE}\nc\cO{\calO}
\nc\tate{\mathrm{Tate}}

\newcommand{\x}{x(\tau)}
\newcommand{\y}{y(\tau)}

\newcommand{\SL}{{\text {\rm SL}}}
\newcommand{\GL}{{\text {\rm GL}}}
\newcommand{\sm}[4]{\left(\begin{smallmatrix}#1&#2\\ #3&#4 \end{smallmatrix}
\right)}
\def\H{\mathbb{H}}
\nc\Qm{\Q(\zeta_m)}
\nc\chip{\chi_\mathfrak{p}}

\nc\p{\mathfrak{p}}
\nc\Fq{\bbF_q}
\nc\chiN{\chi_{2N}}
\nc\E{\mathcal{E}^N_t}

\theoremstyle{definition}

\theoremstyle{remark}

\numberwithin{equation}{section}

\begin{document}

\title{ Modular forms, de Rham cohomology and congruences}


\author{Matija Kazalicki}
\address{Department of Mathematics\\
University of Zagreb\\
Bijenicka cesta 30\\
Zagreb, Croatia\\}

\email{mkazal@math.hr}

\author{Anthony J.~Scholl}
\address{Department of Pure Mathematics and Mathematical Statistics\\
Centre for Mathematical Sciences\\
Wilberforce Road\\
Cambridge CB3 0WB\\}

\email{a.j.scholl@dpmms.cam.ac.uk}

\subjclass[2010]{Primary 14F40, 11F33, 11F80 }

\date{}

\maketitle

\begin{abstract}
  In this paper we show that Atkin and Swinnerton-Dyer type of
  congruences hold for weakly modular forms (modular forms that are
  permitted to have poles at cusps). Unlike the case of original
  congruences for cusp forms, these congruences are nontrivial even
  for congruence subgroups. On the way we provide an explicit
  interpretation of the de Rham cohomology groups associated to
  modular forms in terms of ``differentials of the second kind''. As
  an example, we consider the space of cusp forms of weight 3 on a
  certain genus zero quotient of Fermat curve $X^N+Y^N=Z^N$. We show
  that the Galois representation associated to this space is given by
  a Grossencharacter of the cyclotomic field $\Q(\zeta_N)$. Moreover,
  for $N=5$ the space does not admit a ``$p$-adic Hecke eigenbasis''
  for (non-ordinary) primes $p\equiv 2,3 \pmod{5}$, which provides a
  counterexample to Atkin and Swinnerton-Dyer's original speculation
  \cite{ASD,KIB,LLY}.
\end{abstract}


\section{Introduction}

In \cite{ASD}, Atkin and Swinnerton-Dyer described a remarkable family
of congruences they had discovered, involving the Fourier coefficients
of modular forms on noncongruence subgroups. Their data suggested (see
\cite{LLY} for a precise conjecture) that the spaces of cusp forms of
weight $k$ for a noncongruence subgroup, for all but finitely many
primes $p$, should possess a $p$-adic Hecke eigenbasis in the sense
that Fourier coefficients $a(n)$ of each basis element satisfy
\[
a(pn)-A_p a(n)+\chi(p)p^{k-1}a(n/p)\equiv 0 \pmod{
  p^{(k-1)(1+\ord_p(n))}},
\]
where $A_p$ is an algebraic integer and $\chi$ is a Dirichlet
character (depending on the basis element, but not on $n$). This
congruence relation is reminiscent of the relation between Fourier
coefficients of Hecke eigenforms for congruence subgroups (which is
surprising since there is no useful Hecke theory for modular forms on
noncongruence subgroups).

Following work by Cartier \cite{Ca}, Ditters \cite{Di} and Katz
\cite{Ka}, the second author proved a substantial part of these
congruences in \cite{SASD}. There remain various questions concerning
the optimal shape of these congruences in the case when the dimension
of the space of cusp forms is greater than one, see
\cite{ALL,LLY,Long}.

In this paper we show that similar congruences (also initially
discovered experimentally) hold for weakly modular forms (that is,
modular forms which are permitted to have poles at cusps). Unlike the
case of Atkin--Swinnerton-Dyer's original congruences for cusp forms,
these congruences are nontrivial even for congruence subgroups
(because the Hecke theory of weakly modular forms is not so good).
The simplest case is the weakly modular form of level 1 and weight 12
\begin{align*}
E_4(z)^6/\Delta(z) &- 1464 E_4(z)^3 = q^{-1} + \sum_{n=1}^\infty
a(n)q^n \\
&= q^{-1} - 142236q + 51123200q^2 + 39826861650q^3 + \cdots
\end{align*}
For every prime $p\ge11$ and integer $n$
with $p^s | n$, its coefficients satisfy the congruence
\[
a(np)-\tau(p)a(n)+p^{11}a(n/p) \equiv 0 \pmod {p^{11s}}.
\]
where $\tau(n)$ is Ramanujan's function. (Note that the coefficients
$a(n)$ grow too rapidly to satisfy any multiplicative identities.)
These and other examples may be found in \S\ref{sec:ex} below.

In the second part of the paper we consider, for an odd integer $N$,
the space of weight 3 cusp forms on a certain genus zero quotient of
Fermat curves $X^N+Y^N=Z^N$.  These cusp forms are CM forms in the
sense that the Galois representation associated to them is given by a
Grossencharacter of the cyclotomic field $\Q(\zeta_N)$. We show that for
$N=5$ the space of weight 3 cusp forms does not admit a $p$-adic Hecke
eigenbasis for (non-ordinary) primes $p\equiv 2,3 \pmod{5}$. Moreover,
for the better understanding of the congruences arising from the
action of Frobenius endomorphism in this situation, we define certain
weakly modular forms, and prove some congruences for them. For more
details see \S\ref{sec:cong}.

In \cite{SASD} congruences were obtained by embedding the module
of cusp forms of weight $k$ (on a fixed subgroup $\Gamma$) into a de
Rham cohomology group $DR(X,k)$, where $X$ is the modular curve
associated to $\Gamma$. This cohomology group is the de Rham
realisation of the motive \cite{SMF} associated to the relevant space
of modular forms. At a good prime $p$, crystalline theory endows
$DR(X,k)\otimes\Z_p$ with a Frobenius endomorphism, whose action
on $q$-expansions can be explicitly computed, and this gives rise to
the Atkin--Swinnerton-Dyer congruences. (See the introduction of
\cite{SASD} for more explanation.) Here we observe that there is an
simple description of $DR(X,k)$ in terms of  ``forms of the
second kind''.  Curiously, such a description does not appear to be
explicitly given anywhere in the literature (although it is implicit
in Coleman's work on $p$-adic modular forms).  The period isomorphism
is particularly transparent in this interpretation.

\section{Summary of theoretical results}\label{sec:summary}

Let $\Gamma\subset SL_2(\Z)$ be a subgroup of finite index. We choose
a number field $K=K_\Gamma\subset\C$ and a model $X_K$ over $K$ for
the compactified modular curve $\Gamma\backslash\frakH^*$ such that:
\begin{itemize}
\item the $j$-function defines a morphism $\pi_K\colon X_K \to
  \bbP^1_K$; and
\item the cusp $\infty\in \Gamma\backslash\frakH^*$ is a
  rational point of $X_K$.
\end{itemize} 
Let $m$ be the width of the cusp $\infty$. Then the completed local
ring $\widehat{\calO_{X,\infty}}$ equals $K[[t]]$ for some $t$ with $\delta
t^m=q$, with $\delta\in K^*$.

Let $X_K^o\subset X_K$ be the complement of the points where the
covering $\frakH\to X_K(\C)$ is ramified.  On $X_\C^o$ we have the
standard line bundle $\sheafom_\C$, such that modular forms of weight
$k$ are sections of $\sheafom_\C^{\otimes k}$, and the canonical
isomorphism $\psi_\C\colon \sheafom_\C^{\otimes 2}\isomarrow
\Omega^1_{X_\C^o}(\log\,\mathrm{cusps})$, identifying forms of weight
2 with holomorphic 1-forms on $X_\C$.  The fibre at infinity has a
canonical generator $\varepsilon_\C\in\sheafom_\C(\infty)$.  If
$-1\notin\Gamma$ we also assume that this structure comes from a
triple $(\sheafom_K,\psi_K,\varepsilon_K\in \sheafom_K(\infty))$ on
$X_K^o$.

We choose a finite set $S$ of primes of $K$, and write $R=\mathfrak{o}_{K,S}$,
satisfying:

\begin{itemize}
\item $6m$ and $\delta$ are in $R^*$; 
\item there exists a smooth projective curve $X/R$ with
  $X_K=X\otimes_RK$, and $\pi_K$ extends to a finite morphism
  $\pi\colon X \to \bbP^1_R$ which is \'etale away from
  $j\in\{\infty,0,1728\}$; 
\item if $-1\notin\Gamma$, $(\sheafom_K,\psi_K,\varepsilon_K)$ extends
  to a triple $(\sheafom,\psi, \varepsilon)$ on $X^o$, with
  $\sheafom(\infty)=R\varepsilon$.
\end{itemize}

Any modular or weakly modular form on $\Gamma$ has a Fourier expansion
at $\infty$ which lies in $\C((q^{1/m}))=\C((t))$.  For any subring
$R'$ of $\C$ containing $R$, and any $k\ge2$, let $S_k(\Gamma,R')$,
$M_k(\Gamma,R')$ be the $R'$-modules of cusp (resp.~modular) forms on
$\Gamma$ of weight $k$ whose Fourier expansions at $\infty$ lie in
$R'[[t]]$. Standard theory shows that $S_k(\Gamma,R)$, $M_k(\Gamma,R)$
are locally free $R$-modules and that, for any $R'$,
\[
S_k(\Gamma,R')=S_k(\Gamma,R)\otimes_R R',\quad 
M_k(\Gamma,R')=M_k(\Gamma,R)\otimes_R R'
\]
For any integer $s$, denote by $M_s^{\wk}(\Gamma,R')$ the $R'$-module
of weakly modular forms (meromorphic at all cusps) of weight $s$ whose
Fourier expansions at $\infty$ lie in $R'((t))$, and let
$S_s^{\wk}(\Gamma,R')$ be the submodule consisting of those $f\in
M^{\wk}_s(\Gamma,R')$ whose constant term at each cusp vanishes.

It is well known that if $k\ge 2$ there is a linear map
\[
\partial^{k-1} \colon M_{2-k}^{\wk}(\Gamma,\C) \to
S_{k}^{\wk}(\Gamma,\C)
\]
which on Fourier expansions (at any cusp) is given by
$(q\,d/dq)^{k-1}$. Consequently $\partial^{k-1}$ maps
$M^{\wk}_{2-k}(\Gamma,R')$ into $S^{\wk}_k(\Gamma,R')$.

\begin{defn}
  Suppose $K\subset K'\subset \C$. Define for $k\ge 2$
  \[
  DR(\Gamma,K',k)=\frac{S^{\wk}_k(\Gamma,K')}{\partial^{k-1}(
    M^{\wk}_{2-k}(\Gamma,K'))}
  \]
  and
  \[
  DR^*(\Gamma,K',k)=\frac{ M^{\wk}_k(\Gamma,K')}{\partial^{k-1}(
    M^{\wk}_{2-k}(\Gamma,K'))}
  \]
\end{defn}
It is clear that for every $K'$, $DR(\Gamma,K',k) =
DR(\Gamma,K,k)\otimes_{K}K'$, and similarly for $DR^*$.

If $R\subset R'\subset\C$ and $f\in M_k^{\wk}(\Gamma,R')$, the
conditions on $S$ imply that the Fourier coefficients of $f$ at any
cusp are integral over $R'$. Write the Fourier expansion of $f$ at a
cusp $z$ of width $m$ as
\[
\tilde{f}_z = \sum_{n\in \Z} a_n(f,z) q^{n/m}.
\]

\begin{defn}
  Let $f\in M_k^{\wk}(\Gamma,R')$. We say that $f$ is \emph{weakly
    exact} if, at each cusp $z$ of $\Gamma$, and for each $n<0$,
  $n^{-1}a_n(f,z)$ is integral over $R'$.  We write
  $M_k^{\wx}(\Gamma,R')$ for the $R'$-module of weakly exact modular
  forms and $S_k^{\wx}(\Gamma,R')$ for the submodule of weakly exact
  cusp forms.
\end{defn}
It is clear that $\partial^{k-1}(M_{2-k}^{\wk}(\Gamma,R') \subset
S_k^{\wx}(\Gamma,R')$.
\begin{defn}
  Define for $k\ge 2$
  \[
  DR(\Gamma,R',k)=\frac{ S^{\wx}_k(\Gamma,R')}{\partial^{k-1}(
    M^{\wk}_{2-k}(\Gamma,R'))}
  \]
  and
  \[
  DR^*(\Gamma,R',k)=\frac{ M^{\wx}_k(\Gamma,R')}{\partial^{k-1}(
    M^{\wk}_{2-k}(\Gamma,R'))}
  \]
\end{defn}
If $R'\supset\Q$ this obviously agrees with our earlier definition.

In \S\ref{sec:rev}, \S\ref{sec:2nd}, and \S\ref{sec:qexp} we will prove that these groups enjoy the following properties.
\begin{itemize}
\item The $R$-modules $DR(\Gamma,R,k)$ and $DR^*(\Gamma,R,k)$ are
  locally free, and for every $R'\supset R$ we have
  \[
  DR(\Gamma,R',k)=DR(\Gamma,R,k)\otimes_R R',\quad 
  DR^*(\Gamma,R',k)=DR^*(\Gamma,R,k)\otimes_R R'
  \]
\item There exists for each $k\ge 2$ a commutative diagram with exact rows
  \[
  \begin{CD}
    0 @>>> S_k(\Gamma,R) @>>> DR(\Gamma,R,k) @>>> S_k(\Gamma,R)^\vee
    @>>> 0
    \\
    @. \bigcup && \bigcup && @| \\
    0 @>>> M_k(\Gamma,R) @>>> DR^*(\Gamma,R,k) @>>> S_k(\Gamma,R)^\vee
    @>>> 0
  \end{CD}
  \]
  in which all the inclusions are the natural ones.
\item Suppose that $p$ is prime, and that for some embedding
  $\Z_p\inject{}\C$, we have $R\subset\Z_p$. Then there are canonical
  compatible endomorphisms $\phi_p$ of $DR(\Gamma,\Z_p,k)$,
  $DR^*(\Gamma,\Z_p,k)$. The characteristic polynomial $H_p(T)$ of
  $\phi_p$ on $DR(\Gamma,\Z_p,k)$ has rational integer coefficients,
  and its roots are $p^{k-1}$-Weil numbers. Moreover
  \[
  H_p(T)=\text{(constant)}T^{2d_k}H_p(1/p^{k-1}T)
  \]
  where $d_k=\dim S_k(\Gamma)$.

  The characteristic polynomial of
  $\phi_p$ on $DR^*(\Gamma,\Z_p,k)/DR(\Gamma,\Z_p,k)$ has integer
  coefficients and its roots are of the form $p^{k-1}\times\text{(root
    of unity)}$.
\item Still assume that $R\subset\Z_p$. There is a unique $\gamma_p\in
  1+p\Z_p$ such that $\gamma_p^m=\delta^{p-1}$. Let $\tilde\phi_p$ be
  the endomorphism of $\Z_p((t))$ given by
  \[
  \tilde\phi_p\colon \sum a_nt^n \mapsto p^{k-1}\sum
  a_n\gamma_p^nt^{np}.
  \]
  Then the diagram
  \[
  \begin{CD}
    DR^*(\Gamma,\Z_p,k) @>>> \dfrac{\Z_p((t))}{\partial^{k-1}\left(\Z_p((t))\right)}
    \\
    @V{\phi_p}VV @VV{\tilde\phi_p}V \\
    DR^*(\Gamma,\Z_p,k) @>>> \dfrac{\Z_p((t))}{\partial^{k-1}\left(\Z_p((t))\right)}
  \end{CD}
  \]
  commutes.
\item Write $\langle k-1\rangle=\inf\{\ord_p(p^j/j!) \mid j\ge k-1\}$,
  and let
  \begin{align*}
    DR^*(\Gamma,\Z_p,k)^{(p)} &=  M_k(\Gamma,\Z_p)+p^{\langle k-1\rangle}
    DR^*(\Gamma,\Z_p,k)\\ &\subset \frac{ M^{\wx}_k(\Gamma,R')}{p^{\langle k-1\rangle}
      \partial^{k-1}(M^{\wk}_{2-k}(\Gamma,R'))}
  \end{align*}
  Then $\phi_p$ preserves $DR^*(\Gamma,\Z_p,k)^{(p)}$ and the diagram
  \[
  \begin{CD}
    DR^*(\Gamma,\Z_p,k)^{(p)} @>>> \dfrac{\Z_p((t))}{p^{\langle k-1\rangle}
\partial^{k-1}\left(\Z_p((t))\right)}
    \\
    @V{\phi_p}VV @VV{\tilde\phi_p}V \\
    DR^*(\Gamma,\Z_p,k)^{(p)} @>>> \dfrac{\Z_p((t))}{p^{\langle k-1\rangle}
\partial^{k-1}\left(\Z_p((t))\right)}
  \end{CD}
  \]
  commutes.
\end{itemize}

\noindent
\emph{Congruences}

We continue to assume that $R\subset \Z_p$. Let $\frako=\frako_F$ for
a finite extension $F/\Q_p$. Extend $\phi_p$ to a $\frako$-linear
endomorphism of $DR^*(\Gamma,\frako,k)$. Let $f\in
M_k^{\wx}(\Gamma,\frako)$, with Fourier expansion at infinity
\[
\tilde f=\sum_{n\in Z}a(n) q^{n/m} = \sum_{n\in\Z} b(n) t^n,\quad b(n)\in\frako.
\]
Let $H=\sum_{j=0}^rA_jT^j\in \frako[T]$ such that the image of $f$ in
$DR^*(\Gamma,\frako,k)$ is annihilated by $H(\phi_p)$. 

\begin{thm}\label{thm:long}
  (i) The coefficients $a(n)$ satisfy the congruences: if $n\in\Z$ and
  $p^s|n$ then
  \[
  \sum_{j=0}^r p^{(k-1)j}A_ja(n/p^j) \equiv 0 \pmod{p^{(k-1)s}}.
  \]
  (ii) If moreover $f\in M_k(\Gamma,\frako)$ then these congruences hold
  mod $p^{(k-1)s+{\langle k-1\rangle}}$.
\end{thm}
Here the left hand side is interpreted as
\[
\delta^{-n/m}\sum_{j=0}^r p^{(k-1)j}A_j\gamma_p^{n(p^j-1)/(p-1)}
b(n/p^j) \in \delta^{-n/m}
\]
which is the product of a unit an an element of $\frako$, and we adopt
the usual convention that $a(n)=b(n)=0$ is $n\notin\Z$
(cf.~\cite[Thm. 5.4]{SASD}). Part (ii) is one of the main results of \cite{SASD}.
\begin{proof}
  The properties above show that 
  \[
  \sum c_nt^n \defeq H(\tilde\phi)(\tilde f)\in \partial^{k-1}\left(\frako((t))\right)
  \]
  or equivalently that for every $n\in\Z$, $c_n \in
  n^{k-1}\frako$. Applying $H(\tilde\phi)$ to $\tilde f$ term-by-term,
  one obtains the congruences (i). If $f\in M_k(\Gamma,\frako)$ then 
  $H(\tilde\phi)(\tilde f)\in  p^{\langle
    k-1\rangle}\Im(\partial^{k-1})$, giving the stronger congruences (ii).
\end{proof}

\section{First examples}\label{sec:ex}

Under the hypotheses of Theorem \ref{thm:ASD}, suppose that
$\dim S_k(X\otimes\Q)=1$ and that $f\in S^{\wx}_k(X)$. Then the
characteristic polynomial of $\phi_p$ on $DR(X\otimes\Z_p,k)$ is of
the form
\[
H_p(T)=T^2-A_pT+p^{k-1},\quad A_p\in\Z
\]
The congruences \eqref{eq:ASDcong} then take the form
\begin{equation}
  \label{eq:3term}
  a(np) \equiv A_p a(n) - p^{k-1}a(n/p)\mod{p^{(k-1)s}}\quad\text{if $p^s|n$}
\end{equation}
Consider the weak cusp form of level one and weight 12
\begin{align*}
f&=E_4(z)^6/\Delta(z) - 1464 E_4(z)^3 .
\end{align*}
We cannot directly apply the theorem to $f$, since the modular curve
of level $1$ does belong to the class of $X$ considered in
\S\ref{sec:rev}. We can get round this in the usual way (cf.~part (b)
proof of \cite[5.2]{SASD}): take $X=X'=X(3)$ for some auxiliary
integer $N\ge 3$, and define $DR(X(1)\otimes\Z[1/6],k) =
DR(X(3),k)^{GL(2,\Z/3\Z)}$, which is then a free $Z[1/6]$-module of
rank 2. For each $p>3$, $DR(X(1)\otimes\Z_p,12)$ is annihilated by
$H_p(\phi) =\phi^2-\tau(p)\phi+p^{11}$, and one recovers, for
$p\ge11$, the congruences of the introduction. (With more care we
could get congruences for small primes as well.)  We also note that
for this example, and others on congruence subgroups, one could
replace the operator $H_p(\phi)$ with $T_p-\tau(p)$ where $T_p$ is the
Hecke operator acting on $DR(X(1)\otimes\Z_p,12)$ (defined using
correspondences in the usual way) and thereby avoid recourse to
crystalline theory.

As a further example, consider the following (weakly) modular forms of
weight 3 for noncongrence subgroup $\Phi_0(3)$ (defined in
\S\ref{sec:Fermat} below):
\begin{align*}
f_1(\tau)&= \eta(\tau/2)^\frac{4}{3}\eta(\tau)^{-2}\eta(2\tau)^\frac{20}{3}\\&=\sum c_1(n)q^\frac{n}{2}=q^{\frac{1}{2}}-\frac{4}{3}q^\frac{2}{2}+\frac{8}{9} q^\frac{3}{2}-\frac{176}{81} q^{\frac{4}{2}}+\cdots \in S_3(\Phi_0(3)),\\
\\
f_2(\tau)&= \eta(\tau/2)^\frac{20}{3}\eta(\tau)^{-10}\eta(2\tau)^\frac{28}{3}\\&=\sum c_2(n)q^\frac{n}{2}=q^\frac{1}{2}-\frac{20}{3}q^\frac{2}{2}+\frac{200}{9} q^\frac{3}{2}-\frac{4720}{81} q^\frac{4}{2}+ \cdots \in S^{\wk}_3(\Phi_0(3)).\\
\end{align*}
(Although $f_2$ is holomorphic at $\infty$, there is another cusp at
which it has a pole.)  From Corollary \ref{cor:two} it follows that
for a prime $p \equiv 2 \bmod{3}$, there exist $\alpha_p,\beta_p \in
\Z_p$ such that if $p^s|n$ then
\begin{align*}
c_1(pn) &\equiv \alpha_p c_2(n) \bmod{p^{2(s+1)}},\\
c_2(pn) &\equiv \beta_p c_1(n) \bmod{p^{2(s+1)}}.
\end{align*}
Moreover $\alpha_p \beta_p = p^2$, and $\ord_p(\alpha_p)=2$.

If $p \equiv 1 \bmod{3}$, then for some $\alpha_p \in \Z_p$ ($\ord_p(\alpha_p)=2$)
\begin{align*}
c_1(pn) &\equiv \frac{p^2}{\alpha_p} c_1(n) \bmod{p^{2(s+1)}},\\
c_2(pn) &\equiv \alpha_p c_2(n) \bmod{p^{2(s+1)}}.
\end{align*}
For any $p>3$ we have 
\[
c_2(pn)-A_p c_2(n)+\chi_3(p)p^2c_2(n/p) \equiv 0\bmod{p^{2s}}
\quad\text{if $p^s|n$},
\]
where $A_p$ is the $p$-th Fourier coefficient of a certain CM newform
in $S_3(\Gamma_1(12))$, and $\chi_3$ is Dirichlet character of
conductor 3 (and $H_p(T)=T^2-A_pT+\chi_3(p)p^2$).

\section{Review of  \cite{SASD}}
\label{sec:rev}

Let $R$ be a field or Dedekind domain of characteristic zero.
In this section we will work with modular curves over $R$. Let  $X$
be a smooth projective curve over $R$, whose fibres need not be
geometrically connected, equipped with a finite morphism $g \colon
X \to X'$, whose target $X'$ is a modular curve for a representatable moduli
problem. In practice we have in mind for $X'$ the basechange from
$\Z[1/N]$ to $R$ of one of the following
curves:
\begin{itemize}
\item[(i)] $X_1(N)$ (for some $N\ge 5$), the modular curve over $\Z[1/N]$
  parameterising (generalised) elliptic curves with a section of order
  $N$;
\item[(ii)] $X(N)$ (for some $N\ge 3$), parameterising elliptic
  curves with a full level $N$ structure $\alpha\colon (\Z/N)^2 \to
  E$,
\item[(iii)] $X(N)^{\arith}$ (for some $N\ge 3$), parameterising
  elliptic curves with ``arithmetic level $N$ structure of determinant
  one'' $\alpha \colon \Z/N\times\mmu_N \to E$
\end{itemize}
and we will limit ourselves to these cases, although most things
should work if $X'$ is replaced by some other modular curve (perhaps
for an ``exotic'' moduli problem).

We let $Y'\subset X'$ be the open subset parameterising true elliptic
curves, and $Z'\subset X'$ the complementary reduced closed subscheme
(the cuspidal subscheme).  We make the following hypotheses
on the morphism $g$:
\begin{quotation}
  (A) $g \colon X \to X'$ is \'etale over $Y'$
  \\
  (B) $\Gamma(X,\calO_X)=K$ is a field.
\end{quotation}
We write $Y$, $Z$ for the (reduced) inverse images of $Y'$, $Z'$ in
$X$, and $j\colon Y \inject{} X$ for the inclusion.

A \emph{cusp} is a connected component $z\subset Z$. The hypotheses
imply (by Abhyankar's lemma) that $g$ is tamely ramified along
$Z'$. We have $z=\Spec R_z$, where $R_z/R$ is finite and \'etale. One
knows that a formal uniformising parameter along a cusp of $X'$ may be
taken to be $q^{1/m}$ for some $m|N$, and we may choose therefore a
parameter $t_z\in \widehat{\calO_{X,z}}$ such that
$\delta_zt_z^{m_z}=q$ for some $m_z\ge 1$, $\delta_z\in
R_z^*$. Moreover $m_z$ (the \emph{width} of the cusp $z$) is
invertible in $R$.

Because $Y'$ represents a moduli problem, there is a universal
elliptic curve $\pi\colon E' \to Y'$, which in each of the cases
(i--iii) extends to a stable curve of genus one $\bar\pi\colon \bar E'
\to X'$, with a section $e\colon X'\to \bar E'$ extending the zero
section of $E'$. We let $\sheafom_{X'} = e^*\Omega^2_{\bar E'/X'}$ be
the cotangent bundle along $e$, and $\sheafom_X$ its pullback to $X$.

If $U$ is any $R$-scheme we shall simply write $\Omega^1_U$ for the
module of relative differentials $\Omega^!_{U/R}$. 

The module of ($R$-valued) modular forms of weight $k\ge 0$ on $X$ is by definition
\[
M_k(X) = H^0(X,\sheafom_X^{\otimes k}).
\]
There is a well-known canonical ``Kodaira--Spencer'' isomorphism
\[
KS(X') \colon \sheafom_{X'}^{\otimes 2} \isomarrow \Omega^1_{X'}(\log Z').
\]
Hypothesis (A) implies that $g^*\Omega^1_{X'}(\log Z')
=\Omega^1_X(\log Y)$, and therefore $KS(X')$ pulls back to give an isomorphism
\[
KS(X) \colon  \sheafom_{X}^{\otimes 2} \isomarrow \Omega^1_{X}(\log Z).
\]
One therefore has
\[
M_k(X) = H^0(X,\sheafom_X^{\otimes k-2}\otimes\Omega^1_{X}(\log Z))
\]
and the submodule of cusp forms is
\[
S_k(X) = H^0(X,\sheafom_X^{\otimes k-2}\otimes\Omega^1_{X}).
\]
Serre duality then gives a canonical isomorphism of free $R$-modules
\[
S_k(X)^\vee \isomarrow H^1(X,\sheafom_X^{\otimes 2-k}).
\]
The relative de Rham cohomology of the family $E' \to Y'$ is
a rank 2 locally free sheaf $\cE_{Y'}=R^1\pi_*\Omega^*_{E'/Y'}$, which
carries an integrable connection $\nabla$. Denote by $\sheafom_Y$,
$\cE_Y$ the pullbacks of $\sheafom_{Y'}$, $\cE_{Y'}$ to $Y$.

There is a canonical extension (in the sense of \cite{DeED}) of
$(\cE_{Y'},\nabla)$ to a locally free sheaf $\cE_{X'}$ with
logarithmic connection
\[
\nabla\colon \cE_{X'} \to \cE_{X'} \otimes \Omega^1_{X'}(\log Z')
\]
whose residue map $\Res_\nabla$ --- defined by the commutativity of the square
\[
\begin{CD}
\cE_{X'} @>{\nabla}>> \cE_{X'}\otimes\Omega^1_{X'}(\log Z') \\
@V{(-)\otimes 1}VV @VV{id\otimes\Res_{Z'}}V \\
\cE_{X'}\otimes \cO_{Z'} @>{\Res_\nabla}>> \cE_{X'}\otimes \cO_{Z'} 
\end{CD}
\]
--- is nilpotent. The canonical extension may be described explicitly
using the Tate curve: in the cases (i--iii), each cusp $z\subset Z'$
is the spectrum of a cyclotomic extension $R'=R[\zeta_M]$ (for some
$M|N$ depending on $z$). The basechange of $E'$ to $R'((q^{1/m}))$ via the $q$-expansion
map is canonically isomorphic to the pullback of the Tate curve
$\tate(q)/\Z[1/N]((q^{1/m}))$, and there is a canonical basis
\begin{gather*}
H^1_{\dr}(\tate(q)/\Z[1/N]((q^{1/m}))) = \Z[1/N]((q^{1/m}))\cdot\ocan \oplus
\Z[1/N]((q^{1/m}))\cdot\xcan
\\
\nabla(\ocan)=\xcan\otimes dq/q,\quad \nabla(\xcan)=0
\end{gather*}
for the de Rham cohomology of the Tate curve. The canonical extension
of $\cE_{Y'}$ to $X'$ is then the unique extension for which, at each
cusp $z$ as above, $\widehat\cE_{X',x}$ is generated by $\ocan$ and
$\xcan$.  In particular, in the basis $(\ocan,\xcan)$ the residue map at
a cusp $z$ of width $m$ has matrix
\[
\Res_{\nabla,z}=
\begin{pmatrix}0 & 0\\ m & 0
\end{pmatrix}.
\]
We write $\sheafom_X$, $\cE_X$ for the pullbacks of
$\sheafom_{X'}$, $\cE_{X'}$ to $X$.  Since the residues are nilpotent,
$\cE_X$ is equal to the canonical extension of $\cE_Y$.

The Hodge filtration of $\cE_{Y'}$ extends to give a short exact
sequence
\[
\begin{CD}
0 @>>> F^1\cE_X = gr^1_F\cE_X = \sheafom @>>> F^0 = \cE_X @>>> \sheafom^\vee @>>> 0
\end{CD}
\]
and the Kodaira-Spencer map is obtained (by tensoring with $\sheafom$) from
the composite
\[
\sheafom_X \inject{} \cE_X \mapright{\nabla} \cE_X\otimes\Omega^1_{X}(\log Z) \to
\sheafom_X^\vee \otimes\Omega^1_{X}(\log Z)
\]
In \cite{SASD}, some de Rham cohomology groups associated to modular
forms were defined. Define, for an integer $k\ge2$,
\begin{gather*}
  \Omega^0(\cE_X^{(k-2)})=\cE_X^{(k-2)}\defeq\Sym^{k-2}\cE_X,
  \\
  \Omega^1(\cE_X^{(k-2)})\defeq
  \nabla(\cE_X^{(k-2)})+\cE_X^{(k-2)}\otimes\Omega^1_{X}
  \subset \cE_X^{(k-2)}\otimes \Omega^1_{X}(\log Z)
\end{gather*}
and let 
\[
\nabla^{(k-2)} \colon \Omega^0(\cE_X^{(k-2)}) \to \Omega^1(\cE_X^{(k-2)})
\]
be the $(k-2)$-th symmetric power of the connection $\nabla$.
This makes $\Omega^\bullet(\cE_X^{(k-2)})$ into a complex of locally free
$\cO_X$-modules with $R$-linear maps.  Define
\begin{align}
DR(Y,k)&\defeq H^1(X,\cE_X^{(k-2)}\otimes\Omega_{X}^*(\log Z)),\notag \\
DR(X,k)&\defeq H^1(X,\Omega^\bullet(\cE_X^{(k-2)}))
\end{align}
In the notation of \S2 of \cite{SASD}, $DR(X,k)=L_{k-2}(X,R)$ and $DR(Y,k) =
T_{k-2}(X,R)$. 

The Hodge filtration on $\cE_X^{(k-2)}$ is the symmetric power of the
Hodge filtration $F^\bullet$ on $\cE_X$: its associated graded is
\[
\gr^j_F\cE_X^{(k-2)} = \begin{cases}
  \sheafom_X^{\otimes(k-2-2j)} & \text{if $0\le j\le k-2$}\\
  0&\text{otherwise}
\end{cases}.
\]
Define the filtration $F^\bullet$ on the complex
$\cE_X^{(k-2)}\otimes\Omega_X^\bullet(\log Z)$ by
\[
F^j(\cE_X^{(k-2)}\otimes\Omega_X^i(\log Z)) =
F^{j-i}(\cE_X^{(k-2)})\otimes\Omega_X^i(\log Z).
\]
Then the connection $\nabla^{(k-2)}$ respects $F^\bullet$. On the
associated graded, $\nabla^{(k-2)}$ is $\cO_X$-linear, and if $(k-2)!$
is invertible in $R$, away from the extreme degrees it is an isomorphism:
\begin{gather*}
  \gr_F^0(\cE_Y^{(k-2)}\otimes\Omega_X^\bullet(\log Z)) =
  \sheafom_X^{\otimes 2-k}
  \\
  \gr_F^{k-1}(\cE_Y^{(k-2)}\otimes\Omega_X^\bullet(\log Z)) =
  \sheafom_X^{\otimes k-2}\otimes\Omega^1_X(\log Z)[-1]
  \\
  \gr^j_F\nabla^{(k-2)} \colon \gr^j_F\cE_X^{(k-2)} \isomarrow
  \gr^{j-1}_F\cE_X^{(k-2)}\otimes \Omega_X^1(\log Z)
  \qquad \text{if $0<j< k-1$}
\end{gather*}
In fact, $\gr^j_F\nabla^{(k-2)} = j(KS\otimes id_{\sheafom^{\otimes
    k-2j}})$ if $0<j<k-1$. Therefore from the spectral sequences for
the cohomology of the filtered complexes
\[
(\cE_Y^{(k-2)}\otimes\Omega_X^\bullet(\log Z),
F^\bullet)\quad\text{and}\quad
(\Omega^\bullet(\cE_X^{(k-2)}), F^\bullet)
\]
we obtain a commutative diagram with exact rows 
\[
\begin{CD}
  0 @>>> S_k(X) @>>> DR(X,k) @>>> S_k(X)^\vee @>>> 0 \\
  @. @V{\subset}VV @VVV @| \\
  0 @>>> M_k(X) @>>> DR(Y,k) @>>> S_k(X)^\vee @>>> 0 
\end{CD}
\]
and 
\[
H^j(X,\Omega^\bullet(\cE_X^{(k-2)}))= 
H^j(X,\cE_X^{(k-2)}\otimes\Omega_{X}^\bullet(\log Z))=0\quad
\text{if $j\ne 1$, $k>0$.}
\]
More precisely, there are isomorphisms in the derived category
\begin{align}
\label{eq:complexes}
  \cE_Y^{(k-2)}\otimes\Omega_X^\bullet(\log Z) &= \bigl[\ 
    \sheafom_X^{\otimes 2-k} \mapright{\calD^{k-1}}
    \sheafom_X^{\otimes k-2}\otimes\Omega^1_X(\log Z)
  \ \bigr ]
  \\
  \Omega^\bullet(\cE_X^{(k-2)}) &= \bigl[ \ 
    \sheafom_X^{\otimes 2-k} \mapright{\calD^{k-1}}
    \sheafom_X^{\otimes k-2}\otimes\Omega^1_X
  \ \bigr ]
\end{align}
where $\calD^{k-1}$ is a differential operator which is characterised
by its effect on $q$-expansion:
\[
\calD^{k-1}(f\,\ocan^{2-k}) = \frac{(-1)^k}{(k-2)!}\Bigl(d\frac
  d{dq}\Bigr)^{\!\!k-1}\!\!(f)\,\ocan^{k-2}\otimes \frac{dq}q
\]
(see \cite[proof of 2.7(ii)]{SASD}).

Finally note that from the exact sequence of complexes
\[
0 \mapright{}   \Omega^\bullet(\cE_X^{(k-2)})  \mapright{}
\cE_X^{(k-2)}\otimes\Omega_X^\bullet(\log Z)
\mapright{\Res_Z} \sheafom_X^{\otimes k-2}\otimes\calO_Z 
\mapright{} 0
\] 
we obtain an exact sequence
\[
0 \mapright{} DR(X,k) \mapright{} DR(Y,k) \mapright{\Res} 
\Gamma(Z,\sheafom_X^{\otimes k-2}\otimes\calO_Z)
\mapright{} 0
\]
\section{Modular forms of the second and third kind}
\label{sec:2nd}
For any $k\in\Z$, and any $R$, define
\[
M^{\wk}_k(X)\defeq \Gamma(Y,\sheafom_Y^k),
\] 
the $R$-module of weakly (or meromorphic) modular forms of weight $k$
on $X$. We say that an element of $M^{\wk}_k(X)$ is a weak cusp
form if, at each cusp, its $q$-expansion has vanishing constant
term. Let $S^{\wk}_k(X)\subset M^{\wk}_k(X)$ denote the submodule of
weak cusp forms.

Composing $\calD^{k-1}$ with the Kodaira-Spencer isomorphism we obtain
a $R$-linear map
\[
\theta^{k-1}\colon M^{\wk}_{2-k}(X) \to M^{\wk}_k(X)
\]
which on $q$-expansions is given by $(q\,d/dq)^{k-1}$, and whose image
is contained in $S_k^{\wk}(X)$.

Suppose $R=K$ is a field. Then one knows (cf. \cite{DeED}) that the
restriction map
\[
H^*(X,\cE_X^{(k-2)}\otimes\Omega_X^\bullet(\log Z))
\to
H^*(Y,\cE_Y^{(k-2)}\otimes\Omega_Y^\bullet)
\]
is an isomorphism, and since $Y$ is affine, the cohomology group on
the right can be computed as the cohomology of the complex of groups
of global sections.

We therefore have the following description of the de Rham cohomology
groups as ``forms of the second and third kind'':
\begin{thm}
\label{thm:2ndkindQ}
If $R$ is a field, there exist canonical isomorphisms
\[
DR(Y,k) = \frac {M^{\wk}_k(X)} {\theta^{k-1}(M^{\wk}_{2-k}(X))}
,\qquad
DR(X,k) = \frac {S^{\wk}_k(X)} {\theta^{k-1}(M^{\wk}_{2-k}(X))}
\]
compatible with the inclusions on both sides. The Hodge filtrations on
$DR(Y,k)$ and $DR(X,k)$ are induced by the inclusion $M_k(X)\subset
M_k^{\wk}(X)$.
\end{thm}

\begin{rems}
(i) When $k=2$ we simply recover the classical formulae for the first de
Rham cohomology of a smooth affine curve $Y=X\setminus Z$ over a field
of characteristic zero:
\[
H^1_{\mathrm{dR}}(Y/K)=
\frac{\Gamma(Y,\Omega^1_Y)}{d\left(\Gamma(Y,\cO_Y)\right)}
\]
and for the  complete curve $X$
\[
H^1_{\mathrm{dR}}(X/K)=\frac{\{\text{forms of the 2nd kind on $X$,
    regular on $Y$}\}}{d\left(\Gamma(Y,\cO_Y)\right)}
\]
(ii) Suppose $K=\C$ and $Y(\C)=\Gamma\backslash\mathfrak{H}$ is a
classical modular curve. Then one has a natural isomorphism from
$DR(X,K)$ to Eichler--Shimura parabolic cohomology \cite{Shi} given by periods:
\[
f(z) \mapsto \left( \int_{z_0}^{\gamma(z_0)}  P(z,1) f(z) \, dz \right)_\gamma
\] 
for homogeneous $P\in \C[T_0,T_1]$ of degree $(k-2)$.
\end{rems}

For general $R$, the description given in the theorem needs to be
modified. Since the $R$-modules $DR(Y,k)$ and $DR(X,k)$ are locally
free, and their formation commutes with basechange, restriction to $Y$
induces an injective map
\begin{equation}
\label{eq:2ndmap}
DR(Y,k) \to \frac {M^{\wk}_k(X)} {\theta^{k-1}(M^{\wk}_{2-k}(X))}.
\end{equation}
For each cusp $z\subset Z$, let $R_z=\Gamma(z,\calO_z)$ and let
$t_z\in\widehat{\calO_{X,z}}$ be a uniformising parameter on $X$ along
$z$. Say that $f\in M^{\wk}_k(X)$ is \emph{weakly exact} if
for every cusp $z$, the principal part of $f$ at $z$ is in the image
of $\theta^{k-1}$. Explicitly, if the expansion of $f$ at $z$ is
$\sum a_n t_z^n\otimes \ocan^{\otimes k}$, the condition is that
$a_n\in n^{k-1}R_z$ for every $n< 0$. Let
\[
S^{\wx}(X) \subset M^{\wx}_k(X) \subset M^{\wk}_k(X)
\]
denote the submodules of weakly exact cusp and modular forms,
respectively.

If $g\in M^{\wk}_{2-k}(X)$ then evidently $\theta^{k-1}(g)$ is
weakly exact.

\begin{thm}
\label{thm:2ndkindZ}
For any $R$ the maps \eqref{eq:2ndmap} induce isomorphisms
\[
DR(X,k) = \frac {S^{\wx}_k(X)}{\theta^{k-1}(M^{\wk}_{2-k}(X))},\qquad
DR(X,k) = \frac {S^{\wx}_k(X)}{\theta^{k-1}(M^{\wk}_{2-k}(X))}.
\]
\end{thm}

\begin{proof}
  Let $X_{/Z}=\Spec \widehat{\calO_{X,Z}}$ denote the
  formal completion of $X$ along $Z$, and $Y_{/Z}=X_{/Z}-Z$ the
  complement; thus
  \[
  X_{/Z}=\coprod_z \Spec R_z[[t_z]] \supset Y_{/Z}=\coprod_z \Spec R_z((t_z)).
  \]
  Then $Y\coprod X_{/Z}$ is a faithfully flat affine covering of $X$,
  and so its Cech complex computes the cohomology of any complex of
  coherent $\calO_X$-modules with $R$-linear maps. Applying this to
  the complex \eqref{eq:complexes}, we see that $DR(X,k)$ is the
  $H^1$ of the double complex of $R$-modules:
  \[
  \begin{CD}
    \Gamma(Y_{/Z},\sheafom^{2-k}) @>{\theta^{k-1}}>> \Gamma(Y_{/Z},\sheafom^{k})
    \\
    @AAA @AAA
    \\
    M^{\wk}_{2-k}(X) \oplus \Gamma(X_{/Z},\sheafom^{2-k})  @>{\theta^{k-1}}>>
    S^{\wk}_k(X) \oplus \Gamma(X_{/Z},\sheafom^{k})
  \end{CD}
  \]
  or equivalently the $H^1$ of the complex
  \[
  M^{\wk}_k(X) \mapright{\theta^{k-1}}
  S^{\wk}_k(X) \mapright{\beta}
  \frac{\Gamma(Y_{/Z},\sheafom^{k})}{\Gamma(X_{/Z},\sheafom^{k})+
    \theta^{k-1}\Gamma(Y_{/Z},\sheafom^{2-k})}
  \]
  and $S^{\wx}_k(X)$ is precisely $\ker(\beta)$. Likewise for $DR(Y,k)$.
\end{proof}

\section{$q$-expansions and crystalline structure}
\label{sec:qexp}
Let $z\subset Z$ be a cusp, and write
\[
\partial = q\frac{d}{dq}=m_zt_z\frac{d}{dt_z},
\]
a derivation of $R_z((t_z))$. We have the local expansion maps
\[
\locz \colon DR(X,k) \to \frac{R_z[[t_z]]}{\partial^{k-1}(R_z[[t_z]])},\quad
DR(Y,k) \to \frac{R_z((t_z))}{\partial^{k-1}(R_z((t_z))}
\]
such that the restriction of $f\in DR(X,k)$ to the formal
neighbourhood of $z$ is $\locz(f)\otimes \ocan^{\otimes k}$.

Suppose now that $R=\mathfrak{o}_K$ for a finite unramified extension
$K/\Q_p$, and let $\sigma$ be the arithmetic Frobenius
automorphism of $K$. For each $z$, denote also by $\sigma$ the
Frobenius automorphism of $R_z$ (which is also an unramified extension of
$\Z_p$). By Hensel's lemma there is a unique $\gamma_z$ with
\[
\gamma_z\in 1+pR_z\quad\text{and}\quad \gamma_z^{m_z}=\delta_z^p/\sigma(\delta_z).
\]
The $\sigma$-linear endomorphism $q\mapsto q^p$ of $R((q))$ then
extends to a unique $\sigma$-linear endomorphism of $R_z((t_z))$ whose
reduction is Frobenius, given by
\[
t_z \mapsto \gamma_z t_z^p
\]
Then, as explained in \S3 of \cite{SASD}, there are compatible $\sigma$-linear
endomorphisms $\phi$ of $DR(X,k)$ and $DR(Y,k)$, with the property
that
\begin{equation}
\label{eq:frob}
\locz(f)=\sum a_nt_z^n
\quad\implies\quad
\locz(\phi(f)) = p^{k-1}\sum \sigma(a_n)\gamma_z^nt_z^{np}
\end{equation}
Let us assume that $R=\Z_p$, so that $\phi$ is now linear.  Let
$z\subset Z$ be a cusp with $R_z=\Z_p$. If $f\in M_k^{\wx}(X)$, write
the local expansion of $f$ at $z$ as
\begin{equation}
  \label{eq:exp}
  f = \tilde f \otimes\ocan^{\otimes k},\quad
  \tilde f = \sum b(n) t_z^n
  = \sum a(n) q^{n/m_z},\quad b(n)=\delta_z^{n/m_z}a(n)\in \Z_p.
\end{equation}
Suppose that $H(T)=\sum_{j=0}^rT^j\in \Z_p[T]$ satisfies $H(\phi)(f)=0$ in
$DR(Y,k)$. Then $\locz(H(\phi)f)=0$, which is equivalent to the following
congruences: if $p^s|n$ then
\begin{equation}
  \label{eq:ASDcong}
  \sum_{j=0}^r p^{(k-1)j}A_ja(n/p^j) \equiv 0 \mod{p^{(k-1)s}}.
\end{equation}
Here we follow the usual convention that $a(n)=b(n)=0$ for $n$ not an
integer, and the left hand side is interpreted as
\[
\delta_z^{-n/m_z}\sum_{j=0}^r p^{(k-1)j}A_j\gamma_p^{n(p^j-1)/(p-1)}
b(n/p^j) \in \delta_z^{-n/m_z}\Z_p
\]
cf.~\cite[Thm, 5.4]{SASD}. Putting this together we obtain the following
extension of the ASD congruences to weakly modular forms:
\begin{thm}
  \label{thm:ASD}
  Suppose that $R=\Z[1/M]$ and that $z$ is a cusp with $R_z=R$.
  Let $f\in M^{\wx}_k(X)$, with local expansion at $z$ \eqref{eq:exp}.
  Let $p$ be a prime not dividing $M$ with $p>k-2$, and suppose that 
  the image of $f$ in $DR(Y\otimes\Z_p,k)$ is annihilated by $H(\phi)$
  for some polynomial $H(T)=\sum_{j=0}^rA_jT^j\in \Z_p[T]$. then for
  every integer $n$ the congruences \eqref{eq:ASDcong} hold. 
\end{thm}

\section{Fermat groups and modular forms}
\label{sec:Fermat}
Modular function and modular forms on Fermat curves have been studied by 
D.~Rohrlich \cite{Roh} and T.~Yang \cite{TY}, among others. We follow
here the notation of \cite{TY}.

Let $\Delta$ be the free subgroup of $\SL_2(\Z)$ generated by
the matrices $A:=\sm 1 2 0 1$ and $B:=\sm 1 0 2 1$. One has that
$\Gamma(2)=\{\pm I\}\Delta$. Given a positive integer $N$, the
Fermat group $\Phi(N)$ is defined to be the subgroup of $\Delta$
generated by $A^N$, $B^N$, and the commutator $[\Delta,\Delta]$.
It is known that the modular curve $X(\Phi(N))$ is isomorphic to
the Fermat curve $X^N+Y^N=1$. The group $\Phi(N)$ is a
congruence group only if $N=1,2,4$ and $8$.

Let $N>1$ be an odd integer. Denote by $\Phi_0(N)$ the group generated
by $\Phi(N)$ and $A$. It is a subgroup of $\Delta$ of index $N$ and
genus zero.  (The other two genus zero index $N$ subgroups of $\Delta$
that contain $\Phi(N)$ are generated by $\Phi(N)$ and $AB^{-1}$ and
$B$ respectively.)  The associated modular curve $X(\Phi_0(N))$ is a
quotient of the Fermat curve, and is isomorphic to the curve
\[
v^N=\frac{u}{1-u},
\]
where $u=X^N$ and $v=\frac{X}{Y}$.  Denote by $\H$ the complex upper
half-plane. If $\tau \in \H$ and $q=e^{2\pi i \tau}$, then
\begin{align*}
  \tilde{\lambda}(\tau)&=-\frac{1}{16}q^{-1/2}\prod_{n=1}^\infty
  \left( \frac{1-q^{n-1/2}}{1+q^n} \right)^8,\\
  1-\tilde{\lambda}(\tau)&=\frac{1}{16}q^{-1/2}\prod_{n=1}^\infty
  \left( \frac{1+q^{n-1/2}}{1+q^n} \right)^8
\end{align*}
are modular functions for $\Gamma(2)$. Moreover, they are holomorphic
on $\H$, and $\tilde{\lambda}(\tau)\ne 0,1$ for all $\tau \in \H$. It
follows that there exist holomorphic functions $\tilde{x}(\tau)$ and
$\tilde{y}(\tau)$ on $\H$, such that
$\tilde{x}(\tau)^N=\tilde{\lambda}(\tau)$ and $\tilde{y}(\tau)^N=
1-\tilde{\lambda}(\tau)$, so we have
that
\[
\tilde{x}(\tau)^N+\tilde{y}(\tau)^N = 1
\]
and in fact both $\tilde{x}(\tau)$ and $\tilde{y}(\tau)$ are modular
functions for $\Phi(N).$ We normalize $\tilde{x}(\tau)$ and
$\tilde{y}(\tau)$ by
setting
\[
\x:=(-1)^\frac{1}{N}16^\frac{1}{N}\tilde{x}(\tau) \quad
\textrm{and} \quad \y:=16^\frac{1}{N}\tilde{y}(\tau).
\]
Now, $\x$ and $\y$ have rational Fourier coefficients, and we have
that
\begin{equation}\label{eqx}
\x^N-\y^N= -16.
\end{equation}

For $\gamma= {\sm a b c d} \in \SL_2(\Z)$ and a (weakly) modular form
$f(\tau)$ of weight $k$ define as usual the slash operator
$$(f|\gamma)(\tau):= (c\tau+d)^{-k}f(\gamma \tau).$$
A straightforward calculation \cite[\S2]{TY} shows
\begin{align*}
(x|A)(\tau) &= \zeta_N \x \qquad (x|B)(\tau) = \zeta_N \x,\\
(y|A)(\tau) &= \zeta_N \x \qquad (y|B)(\tau) = \y,
\end{align*}
where $\zeta_N$ is a primitive $N$th root of unity.  
Hence $$t(\tau):=\frac{\x}{\y}$$ is invariant under $\Phi_0(N)$.

The modular curve $X(2)$ has three cusps: $0$, $1$, and
$\infty$. There is one cusp of the curve $X(\Phi_0(N))$ lying above
each of the cusps $0$ and $1$, and $N$ cusps $\infty_1,\ldots ,
\infty_N$ lying above the cusp $\infty$. As functions on
$X(\Phi_0(N))$, $\tilde{\lambda}(\tau)$ and $1-\tilde{\lambda}(\tau)$
have simple poles at $\infty_i$, and they have zeros of order $N$ at
the cusps $0$ and $1$ respectively. The function $t(\tau)$ is
holomorphic on $\H$, nonzero at the cusps above infinity, has a pole
of order one at the cusp $1$, and a zero of order one at the cusp $0$
(so $t(\tau)$ is a Hauptmoduln for $X(\Phi_0(N))$).

Denote by $S_3(\Phi_0(N))$ the space of cusp forms of weight 3 for
$\Phi_0(N)$. It is well known that $\theta_1(\tau):=(\sum_{n \in \Z}
e^{\pi i n^2 \tau})^2$ is a modular form of weight 1 for $\Delta$. It
has a zero at the (irregular) cusp $1$ of order $1/2$.

Let $\Gamma$ be a finite index subgroup of $\SL_2(\Z)$ of genus $g$
such that $-I\notin \Gamma$. For $k$ odd, Shimura \cite[Theorem 2.25]{Shi}
gives the following formula for the dimension of $S_k(\Gamma)$
\[
\dim S_k(\Gamma) =
(k-1)(g-1)+\frac{1}{2}(k-2)r_1+\frac{1}{2}(k-1)r_2+\sum_{i=1}^{j}\frac{e_i-1}{2e_i},
\]
where $r_1$ is the number of regular cusp, $r_2$ is the number of
irregular cusps, and the $e_i$ are the orders of elliptic
points. Since $\Phi_0(N)$ has no elliptic points ($\Delta$ being
free), it follows that $\dim S_3(N)=\frac{N-1}{2}$.

Define 
\begin{equation}
  \label{eq:fi} 
  f_i(\tau):=\theta_1^3(\tau) t^i(\tau)
  \frac{1}{16(1-\tilde{\lambda}(\tau))} 
\end{equation}
for $i=1,2, \ldots N-1$.
The
divisor of $f_i(\tau)$ is
\[
\textrm{div}(f_i) = i(0)+(\frac{1}{2}N-i)(1)+\sum_{j=1}^N (\infty_j).
\]
Hence $\{f_i(\tau)\}$, for $i=1,\ldots, \frac{N-1}{2}$, form a basis
of $S_3(\Phi_0(N))$. If $i=\frac{N+1}{2}, \ldots, N-1$, then
$f_i(\tau)$ has a pole at the cusp $1$, and since the cusp $1$ is
irregular the constant Fourier coefficient is zero. It follows
$f_i(\tau)\in S_3^\wx(\Phi_0(N))$. Since $(t|B)(\tau)=\zeta_N
t(\tau)$, it follows that $(f_i|B)(\tau)=\zeta_N^i f_i(\tau)$.

\section{$\ell$-adic representations}

In this section we define two closely related compatible families of
$\ell$-adic Galois representations of $\GalQ$ attached to the space of
cusp forms $S_3(\Phi_0(N))$. The first family $\rho_{N,\ell}:\GalQ
\mapright{} \GL_{N-1}(\Ql)$ is a $\ell$-adic realisation of the motive
associated to the space of cusp forms $S_3(\Phi_0(N))$ (which we
recall has dimension $(N-1)/2$). It is a special case of second
author's construction from \cite[Section 5]{SASD}. For a more detailed
description see \cite[Section 5]{LLY}.

To describe the second family, consider the elliptic surface fibred
over the modular curve $X(\Phi_0(N))$ defined by the affine equation
\[
\mathcal{E}^N: Y^2= X(X+1)(X+t^N),
\]
together with the map 
\[
h:\mathcal{E}^N \mapright{} X(\Phi_0(N)),
\]
mapping $(X,Y,t) \longmapsto t$. It is obtained from the Legendre
elliptic surface fibred over $X(2)$
\[
\mathcal{E}: Y^2 = X(X-1)(X-\lambda),
\]
by substituting $\lambda = 1-t^N$.  Note that $\lambda$ corresponds to
$\lambda(\tau)=16q^{\frac{1}{2}}-128q+704q^{\frac{3}{2}}+\cdots$, the
usual lambda modular function on $\Gamma(2)$, and we can check
directly that
\[
\lambda(\tau)=1-t(\tau)^N.
\]
The map $h$ is tamely ramified along the cusps and elliptic points so
following \cite[Section 5]{LLY} we may define $\ell$-adic Galois
representation $\rho_{N,\ell}^*:\GalQ \mapright{} \GL_{N-1}(\Ql)$ as
follows: let $X(\Phi_0)^0$ be the complement in $X(\Phi_0)$ of the cusps and
elliptic points. Denote by $i$ the inclusion of $X(\Phi_0)^0$ into
$X(\Phi_0)$, and by $h': \mathcal{E}_N \mapright{} X(\Phi_0(N))^0$ the
restriction of $h$. For a prime $\ell$ we obtain a sheaf
\[
\mathcal{F}_\ell=R^1 h'_*\Ql
\]
on $X(\Phi_0)^0$, and also a sheaf $i_*\mathcal{F}_\ell$ on
$X(\Phi_0)$. The action of $\GalQ$ on the $\Ql$-space
\[
W_\ell = H^1_{et}(X(\Phi_0)\otimes \Qbar, i_*\mathcal{F}_\ell)
\]
defines an $\ell$-adic representation $\rho_{N,\ell}^*:\GalQ
\mapright{} \GL_{N-1}(\Ql)$.

Proposition 5.1 of \cite{LLY} implies that the two representations
$\rho_{N,\ell}^*$ and $\rho_{N,\ell}$ are isomorphic up to a twist by
a quadratic character of $\GalQ$.

\section{Jacobi sums and Gr\"ossencharacters of cyclotomic field}

We review some results of Weil \cite{Weil}.  Let $m>1$ be an integer,
$\zeta_m$ a primitive $m$-th root of unity, and $\mathfrak{p}$ a prime
ideal of $\Qm$ relatively prime to $m$. For any integer $t$ prime to
$m$, let $\sigma_t\in \mathrm{Gal}(\Qm/\Q)$ be the automorphism $\zeta_m
\rightarrow{} \zeta_m^t$. Denote by $q$ the norm of
$\mathfrak{p}$, so that $q\equiv 1 \pmod{m}$. Let $\chip$ be the
$m$-th power residuse symbol: for $x\in \Qm$ prime to
$\mathfrak{p}$, $\chip(x)$ is the unique $m$-th root of unity
such that 
\[
\chip(x) \equiv x^{\frac{q-1}{m}} \pmod{\mathfrak{p}}.
\]
It follows that $\chip:{\Z[\zeta_m]}/{\mathfrak{p}}\cong
\bbF_q \mapright{} \mu_m$ is a multiplicative character of order $m$.

\begin{defn}[Jacobi sums]
  For a positive integer $r$ and $a=(a_1,\ldots, a_r)\in \Z^r$ we
  define
  \[
  J_a(\mathfrak{p}):=(-1)^r \sum_{\substack{x_1+\ldots+x_r \equiv -1
      (\mathfrak{p}) \\ x_1,\ldots, x_r \bmod
      \mathfrak{p}}}\chi_\mathfrak{p}(x_1)^{a_1}\ldots
  \chi_p(x_r)^{a_r},
  \]
  where sum ranges over complete set of representatives of congruence
  classes modulo $\frak{p}$ in $\Qm$. We extend the definition of
  $J_a(\mathfrak{a})$ to all ideals $\mathfrak{a}$ of $\Qm$ prime to
  $m$ by multiplicativity.
\end{defn}

Let $K$ be a number field. $J_K=\prod_{\nu}^\prime K_\nu^*$ its idele
group of $K$. Recally that a \emph{Gro\"ssencharacter} of $K$ is any
continuous homomorphism $\psi\colon J_K \to \C^\times$, trivial on the
group of principal ideles $K^\times\subset J_K$, and that $\psi$ is
unramified at a prime $\p$ if $\psi(\frak{o}_\p^\times)=1$.

Recall also the standard way to view a Gr\"ossencharacter $\psi$ as a
function on the nonzero ideals of $K$, as
follows. Let $\mathfrak{p}$ be a prime of $K$, let $\pi$ be
a uniformizer of $K_\mathfrak{p}$, and let $\alpha_\mathfrak{p}\in J_K$
be the idele with component $\pi$ at the place $\p$ and $1$ at all
other places. One defines
\begin{equation*}
\psi(\p)= 
\begin{cases}
\psi(\alpha_\p) & \text{if } \psi \text{ is unramified at } \p, \\
0 & \text{otherwise}\\
\end{cases}
\end{equation*}
and extends the definition to all nonzero ideals by multiplicativity.
\begin{defn}
  The Hecke $L$-series attached to a Grossencharacter $\psi$ of $K$ is
  given by the Euler product over all primes of $K$
  \[
  L(\psi,s) =
  \prod_\p\left(1-\frac{\psi(\p)}{N(\p)^s}\right)^{-1}.
  \]
\end{defn}

\begin{thm}[Weil, \cite{Weil}]
\label{thm:Weil}
For each $a\ne (0)$ the function $J_a(\mathfrak{a})$ is a
Grossencharacter on $\Qm$ of conductor dividing $m^2$. Its ideal
factorisation is given by the formula
\[
(J_a(\mathfrak{a}))=\mathfrak{a}^{\omega_m(a)},
\]
where 
\[
\omega_m(a)=\sum_{\substack{(t,m)=1\\t \bmod{m}}}\left[
  \sum_{\rho=1}^r \left\langle \frac{t
      a_\rho}{m}\right\rangle\right]\sigma_t^{-1}
\]
and $\langle x \rangle$ denotes the fractional part of a rational number $x$.
\end{thm}
We will need the following technical lemma.

\begin{lem}
\label{lem:technical}
Let $N>1$ be an odd integer, $k$ and $d$ positive integers with $d|N$,
$p\equiv 1 \pmod{N}$ a rational prime, $\mathfrak{p}$ a prime of
$\Z[\zeta_{N/d}]$ above $p$, and $\tilde{\mathfrak{p}}$ a prime of
$\Z[\zeta_{\frac{p^k-1}{d}}]$ above $\mathfrak{p}$. Write
$(p^k-1)/d=2NN'/d.$ Let $J_{(2,N/d)}(\mathfrak{p})$ and
$J_{(2N',NN'/d)}(\tilde{\mathfrak{p}})$ be Jacobi sums associated to the
fields $\Q(\zeta_{N/d})$ and $\Q(\zeta_{\frac{p^k-1}{d}})$ with
defining ideals $2N/d$ and $(p^k-1)/d$ (i.e.~the characters $\chi_\p$ and
$\chi_{\tilde{\p}}$ are of order $2N/d$ and $(p^k-1)/d=2NN'/d$). Then
\[
\left(J_{(2,N/d)}(\p)\right)^{2k}=J_{(2N',NN'/d)}(\tilde{\mathfrak{p}})^2.
\]
\end{lem}
\begin{proof}
  Straightforward calculation shows that the character
  $\chi_{\tilde{\p}}^{N'}$ is the lift of $\chi_\p$,
  i.e.~ $\chi_\p(\textrm{Norm}(x))= \chi_{\tilde{\p}}^{N'}(x)$, for all
  $x \in \Z[\zeta_\frac{p^k-1}{d}]/ \tilde{\p}$, where
  $\textrm{Norm}(x)$ is the norm from $\Z[\zeta_\frac{p^k-1}{d}]/
  \tilde{\p}$ to $\Z[\zeta_\frac{N}{d}]/ \p$.

  Using the factorization of Jacobi sums by Gauss sums (see
  \cite[2.1.3]{BEW}), the lemma then follows directly from the
  Davenport-Hasse theorem on lifted Gauss sums (see
  \cite[11.5.2]{BEW}).
\end{proof}

\section{Traces of Frobenius}
To simplify notation, denote $\mathcal{F}= i_*\mathcal{F}_\ell$. The
Lefschetz fixed point formula and standard facts about elliptic curves
over finite fields gives the following theorem.
\begin{thm}
\label{thm:trace}
$Tr(Frob_q|W_\ell)$ may be computed as follows:
\begin{itemize}
\item[(1)] 
  \[
  Tr(Frob_q|W_\ell)=-\sum_{t\in X(\Phi_0(N))(\bbF_q)} Tr(Frob_q|\mathcal{F}_t).
  \]
\item[(2)] If the fiber $\mathcal{E}^N_t$ is smooth, then
  \[
  Tr(Frob_q|\mathcal{F}_t)=Tr(Frob_q|H^1(\mathcal{E}^N_t,
  \Ql))=q+1-\#\mathcal{E}^N_t(\bbF_q).
  \]
\item[(3)] If the fiber $\mathcal{E}^N_t$ is singular, then 
  \begin{equation*}
    Tr(Frob_q|\mathcal{F}_t)= 
    \begin{cases}
      1 & \text{if the fiber is split multiplicative}, \\
      -1 & \text{if the fiber is nonsplit multiplicative},\\
      0 & \text{if the fiber is additive}.
    \end{cases}
  \end{equation*}
\end{itemize}
\end{thm}

\begin{thm}
\label{thm:gross}
Let $N>1$ be an odd integer, and $\ell$ a prime. The Galois
representations $\rho^*_{N,\ell}$ and $\oplus_{d|N} J_{(2,N/d)}^2$
have the same local factors at every prime $p\nmid 2N\ell$.
\end{thm}
\begin{proof}
  Let $k$ be a positive integer, $p$ be an odd prime, and $q=p^k$ such
  that $q\equiv 1 \pmod{N}$. Let $\chi$ be any character of
  $\Fq^\times$ of order $2N$ (which exists since $q \equiv 1
  \pmod{2N}$). We count the points on the elliptic surface $\mathcal{E}^N$
  (excluding all points at infinity):
\begin{align*}
\#\mathcal{E}^N(\Fq) &= \sum_{t \in \bbF_q} \sum_{x \in \bbF_q} \left( \chi^N(x(x+1)(x+t^N))+1\right) \\
 &= q^2 + \sum_{x\in \Fq} \chi^N(x(x+1))\sum_{t \in \Fq}\chi^N(x+t^N).
\end{align*}
Now
\begin{align*}
	\sum_{t \in \Fq} \chi^N(x+t^N) &=\left[\substack{x_1=t^N\\ x_2=-x-t^N}\right]= \chi^N(x)+\sum_{x_1+x_2 = -x} \sum_{\chi_N \text{ of order $|N$}} \chi_N(x_1)\chi^N(-x_2)\\
	&= \chi^N(x)+\chi^N(-1) \sum_{i=1}^{N-1} \sum_{x_1+x_2 = -x} \chi^{2i}(x_1)\chi^N(x_2).
\end{align*}
Define 
\begin{align*}
J_i(x)&:=\chi^N(-1)\sum_{x_1+x_2=-x}\chi^{2i}(x_1)\chi^N(x_2)\\
	&= \left[\substack{x_1=x_1'\cdot x\\ x_2=x_2'\cdot x}\right]=\chi^N(-1)\sum_{x_1'+x_2'=-1}\chi^{2i}(x)\chi^{2i}(x_1')\chi^N(x)\chi^N(x_2')\\
	&= \chi^N(-1)\chi^N(x)\chi^{2i}(x)\sum_{x_1'+x_2'=-1}\chi^{2i}(x_1'))\chi^N(x_2')\\
  &= \chi^{2i}(x)\chi^N(x)J_i(1).
\end{align*}
Then
\begin{align*}
\#\mathcal{E}^N(\Fq) &=q^2+\sum_{x\in \Fq}\chi^N(x(x+1))\left( \sum_{i=1}^{N-1}J_i(1)\chi^{2i}(x)\chi^N(x)+\chi^N(x) \right) \\
  &= q^2+\sum_{i=1}^{N-1}J_i(1)\left(\sum_{x\in \Fq}\chi^N(-x-1)\chi^N(-1)\chi^{2i}(x)+\sum_{x \in \Fq}\chi^N(x+1)\right)\\
  &= q^2+\sum_{i=1}^{N-1} J_i(1)^2.
\end{align*}
The fibre $\mathcal{E}^N_t$ of the elliptic surface $\mathcal{E}^N$ is
singular if and only if $t=0$ or $t^N=1$. (In the calculation below,
we refer to these $t$ as bad, and the others as good). In the first
case, $\mathcal{E}^N_0:y^2=(x+1)x^2$ has split multiplicative type. In
the second case, $\E:y^2=x(x+1)^2$ is split multiplicative if
$\chi^N(-1)=1$ (or equivalently if $p \equiv 1 \pmod{4}$), and
nonsplit multiplicative if $\chi^N(-1)=-1$. Denote by $M$ the number
of $N$th roots of unity in $\Fq$.  Theorem \ref{thm:trace} implies
\begin{align*}
Tr(Frob_q|W_\ell)&=-\sum_{t\in X(\Phi_0(N))(\bbF_q)} Tr(Frob_q|\mathcal{F}_t)\\
	&= \sum_{t \text{ good}}\#\E(\Fq)-(q+1)\cdot \#\{t\text{ good}\}-\sum_{t\text{ bad}}Tr(Frob_q|\mathcal{F}_t)\\
	&= \#\mathcal{E}^N(\Fq)+\#\{t \text{ good}\}-\sum_{t \text{ bad}}(\#\E(\Fq)-1)\\
	&\qquad\qquad\qquad-(q+1)\#\{t \text{ good}\}-(\chi^N(-1)M+1)\\
	&= q^2 + \sum_{i=1}^{N-1}J_i(1)^2-(q-1)-M(q-\chi^N(-1))\\
	&\qquad\qquad\qquad-q(q-1-M)-(\chi^N(-1)M+1)\\
	&= \sum_{i=1}^{N-1}J_i(1)^2
\end{align*}
Suppose that $p\equiv 1 \pmod{2N}$ (so that $p$ splits completely in
$\Q(\zeta_{2N})$). It is enough to show that
$Tr(Frob_{q}|W_\ell)=\sum_{d|N}\sum_{\p}J_{(2,N/d)}(\p)^{2k}$, where
the second sum is over the primes of $\Q(\zeta_{2N/d})$ lying above
$p$.  Fix $d|N$. For any $\tilde{\p}$ a prime of
$\Q(\zeta_{\frac{p^k-1}{d}})$ above $p$ the residual degree of
$\tilde{\p}$ in $\Q(\zeta_{\frac{p^k-1}{d}})$ is $k$ (since the order
of $p$ in $\left(\Z/d(p^k-1)\Z\right)^\times$ is $k$), hence
$\chi_{\tilde{\p}}$ is a character of $\Fq^\times$ of order
$\frac{p^k-1}{d}$. We can choose $\tilde{\p}$ such
that
\[
J_{(\frac{p^k-1}{N},\frac{p^k-1}{2d})}(\tilde{\p})^2=J_{d}(1)^2.
\]
By Lemma \ref{lem:technical} it follows
\[
J_{d}(1)^2 =J_{(2,N/d)}(\p)^{2k},
\]
where $\p$ is the prime of $\Q(\zeta_{N/d})$ below $\tilde{\p}$.
Since $J_{dj}(1)^2$'s are conjugate to each other for $j=1,\ldots,N/d$
with $(j,N/d)=1$, it follows
that
\[
\sum_{(j,N/d)=1}J_{dj}(1)^2=\sum_{\p\text{ above } p}
J_{(2,N/d)}(\p)^{2k}.
\]
The claim follows after summing over $d|N$. The case $p \not\equiv 1
\pmod{N}$ is proved in a similar way.
\end{proof}

\section{Atkin and Swinnerton-Dyer congruences}\label{sec:cong}

We now apply results of \S\ref{sec:2nd} to obtain congruences of Atkin
and Swinnerton-Dyer type between the Fourier coefficients of the
(weakly) modular forms $f_i(\tau)$. Let $p>3$ be a prime such that $p
\nmid N$. Set $R=\Z_p$, write $X=X(\Phi_0(N))$ and $X'=X(2)$ for the
extensions of the curves considered above to smooth proper curves
over. Let $g:X\mapright{}X'$ be the finite morphism that extends the
quotient map $\Phi_0(N)\backslash \H \mapright{} \Gamma(2)\backslash
\H$ (see proof of \cite[Proposition 5.2 a)]{SASD}). Denote by
$W:=DR(X,3)\otimes \Qpbar$ de Rham space corresponding to this
data. The action of $B={\sm 1 0 2 1}$ on the space of cusp forms
$S_3(\Phi_0(N))$ extends to $W$: for $h^\vee\in S_3(\Phi_0(N))^\vee$
and $f\in S_3(\Phi_0(N))$ we have $(h^\vee|B)(f)=h^\vee(f|B^{-1})$. We
write $W=\oplus_{i=1}^{N-1} W_i$, where $W_i$ is the eigenspace of $B$
corresponding to the eigenvalue $\zeta_N^i$. Since
$(f_i|B)(\tau)=\zeta_N^i f_i(\tau)$ for $i=1, \ldots, N-1$ and
$f_i(\tau)\in S_3^\wx(\Phi_0(N))$, Theorem \ref{thm:2ndkindZ} implies
that $f_i(\tau)\in W_i$. Let $\phi$ be the linear Frobenius
endomorphism of $W$ defined in \S\ref{sec:qexp}.

\begin{prop} 
For $i=1,\ldots, N-1$,
\[
\phi(W_i)\subset W_{i\cdot p \bmod{N}}.
\]
\end{prop}
\begin{proof}
Since $B \phi = \phi B^p$ (see \cite[Section 4.4]{Long}), for $f\in W_i$ we have
\[
\phi(f)|B=\phi((f|B)^p)=\zeta_N^{ip}\phi(f),
\]
and the claim follows.
\end{proof}
Define $\alpha_{i}\in \Z_p$ by $\phi(f_i)=\alpha_i f_{i\cdot p \bmod{N}}$.
\begin{prop}
\label{prop:ord}
\begin{equation*}
\ord_p(\alpha_i) = 
\begin{cases}
2& \text{if } i=1,\ldots \frac{N-1}{2},\\
0 & \text{if } i=\frac{N+1}{2}, \ldots, N-1.
\end{cases}
\end{equation*}
\end{prop}
\begin{proof}
Proposition 3.4 of \cite{SASD} implies
\[
\phi(S_3(\Phi_0(N))) \subset p^2 DR(X,3).
\]
Since the $f_i$ are normalized, it follows that $\ord_p(\alpha_i) \geq
2$, for $i=1,\ldots, \frac{N-1}{2}$. On the other hand, the determinant of
$\phi$ is $\pm p^{2\dim S_3(\Phi_0(N))}=\pm p^{N-1}$, hence
$\ord_p(\alpha_1\cdot \alpha_2\cdot \ldots \cdot \alpha_{N-1})=N-1$
and the claim follows.
\end{proof}

Write $f_i(\tau)=\sum_{j=1}^\infty a_i(j)q^{\frac{j}{2}}$, for
$i=1, \ldots, N-1$. (Note that \eqref{eq:fi} implies $a_i(1)=1$.) From the description of the action of $\phi$
\eqref{eq:frob} on the de Rham space $DR(X,3)$ (and $DR(X,3)^{(p)}$
when $f_i$ is a cusp form, see \S\ref{sec:summary}), we thus obtain:
\begin{cor}
\label{cor:two}
For $i=1,\ldots, N-1$ and any positive integer $j$,
\[
\frac{p^2}{\alpha_i}a_{i}(j)\equiv a_{i\cdot p \bmod{N}}(pj) \pmod
{p^{2(\ord_p(j)+1)}}.
\]
\end{cor}

Suppose $\Gamma$ is a noncongruence subgroup of $\SL_2(\Z)$ of finite
index such that the modular curve $X(\Gamma)$ has a model over $\Q$
(see \S\ref{sec:summary}).  Based on Atkin and Swinnerton-Dyer's
discovery, Li, Long and Yang made the precise conjecture (Conjecture
1.1 of \cite{LLY}) that for each integer $k \geq 2$, there exists a
positive integer $M$ such that for every prime $p\nmid M$ there is a
basis of $S_k(\Gamma)\otimes\Z_p$ consisting of $p$-integral forms $h_i(\tau)$,
$1\leq i \leq d:=\dim S_k(\Gamma)$, algebraic integers $A_p(i)$, and
characters $\chi_i$ such that, for each $i$, the Fourier coefficients
of $h_i(\tau)=\sum_j a_i(j)q^\frac{j}{\mu}$ ($\mu$ being the width of
the cusp at infinity) satisfy the congruence relation
\[
a_i(np)-A_p(i)a_i(n)+\chi_i(p)p^{k-1}a_i(n/p) \equiv 0 \pmod
{p^{(k-1)(1+\ord_p(n)}},
\]
for all $n \geq 1$. 

\begin{thm}
  Let $p$ be any prime congruent to $2$ or $3 \pmod{5}$ be a prime. There is no basis of
  $S_3(\Phi_0(5))\otimes\Z_p$, consisting of $p$-integral forms, satisfying Atkin--Swinnerton-Dyer congruence
  relations for $p$.
\end{thm}
\begin{proof}
  Assume that $\{g_1(\tau),g_2(\tau)\}$ is a basis
  satisfying ASD congruences at $p$. Theorem \ref{thm:gross} implies
  that $\rho_\ell$, the $\ell$-adic representation attached to
  $S_3(\Phi_0(5))$, is isomorphic to the quadratic twist of
  Grossencharacter of $\Q(\zeta_5)$. In particular, since $p$ is inert
  in $\Q(\zeta_5)$, we have that $H_p(T)=T^4\pm p^4$. Theorem
  \ref{thm:long} implies that 
  \[
  b_i(p^4m)\equiv \pm p^4 b_i(m) \pmod{p^6}, \textrm{ for }p \nmid
  m\in \N,
  \]
  where $b_i's$ are Fourier coefficients of $g_i(\tau)$. In particular    	$p|b_i(p^4)$.

  Since $g_i(\tau)$ satisfy ASD congruences, for some algebraic
  integer $A_p(i)$ we have that $b_i(p^k)\equiv A_p(i)b_i(p^{k-1})
  \pmod{p}$, for all $k\ge1$. It follows that $p|b_i(p)$ (if this
  were not the case, this would imply that $p\nmid b_i(p^k)$ for all
  $k\ge 1$). Hence the $p$-th Fourier coefficient of $f_1(\tau)$ and
  $f_2(\tau)$ is divisible by p. However, Proposition \ref{prop:ord}
  implies that either $\phi(f_1(\tau))=\alpha_1 f_2(\tau)$ or
  $\phi(f_2(\tau))=\alpha_2 f_1(\tau)$, and
  $\ord_p(\alpha_1)=\ord_p(\alpha_2)=0$. It follows from Corollary
  \ref{cor:two} that $p$-th Fourier coefficient of $f_1(\tau)$ or
  $f_2(\tau)$ is not divisible by $p$ (since $a_i(1)=1$), which is in contradiction with
  our assumption.
\end{proof}

\begin{rem*}
  J. Kibelbek \cite{KIB} has given an example of a space of weight two
  modular forms that does not admit a basis satisfying Atkin and
  Swinnerton-Dyer congruence relations.
\end{rem*}

\bibliographystyle{amsplain}

\begin{thebibliography}{99}

\bibitem{ALL} A. O. L. Atkin, W.-C. W. Li, L. Long \emph{On Atkin-Swinnerton-Dyer congruence relations(2)}
  Math. Ann.  \textbf{340}  (2008),  no. 2, 335--358.
  
\bibitem{ASD} A. O. L. Atkin, H. P. F. Swinnerton-Dyer, \emph{Modular forms on noncongruence subgroups}
 Combinatorics, (Proc. Sympos. Pure Math., Vol. XIX, Univ. Californis, Los Angeles, 1968), Amer. Math. Soc., 1--25.

\bibitem{BEW} B.C. Berndt, R.J. Evans, K. S. Williams, \emph{Gauss and Jacobi sums},
Canadian Mathematical Society Series of Monographs and Advanced
 Texts. A Wiley-Interscience Publication.
John Wiley \& Sons, Inc., New York,  1998. xii+583 pp.  
 
\bibitem{Ca} P. Cartier, \emph{Groupes formels, fonctions automorphes et fonctions zeta des courbes elliptiques} Actes du Congrès International des Mathématiciens (Nice, 1970), Tome \textbf{2}, 291-–299.

\bibitem{DeED} P. Deligne, \emph{Equations differentielles a points singuliers reguliers}, Lecture Notes in Mathematics, Vol. 163. Springer-Verlag, Berlin-New York, 1970.

\bibitem{Di} B. Ditters, \emph{Sur les congruences d'Atkin et de Swinnerton-Dyer}, C. R. Acad. Sci. Paris Sér. A-B \textbf{282} (1976), no. 19, Ai, A1131–A1134.

\bibitem{Ka} N.M. Katz \emph{$p$-adic properties of modular schemes and modular forms} Modular functions of one variable, III (Proc. Internat. Summer School, Univ. Antwerp, Antwerp, 1972), Lecture Notes in Mathematics,  Springer, Berlin, 1973, Vol. \textbf{350}, 69-–190.


\bibitem{KIB} J. Kibelbek \emph{On Atkin-Swinnerton-Dyer congruence relations for noncongruence subgroups}
 Proc. Amer. Math. Soc., to appear. 

\bibitem{LLY} W.-C. W. Li, L. Long, Z. Yang \emph{On Atkin-Swinnerton-Dyer congruence relations}
 J. Number Theory  \textbf{113}  (2005),  no. 1, 117--148.

\bibitem{Long} L. Long \emph{On Atkin-Swinnerton-Dyer congruence relations (3)}
  J. Number Theory  \textbf{128}  (2008),  no. 8, 2413--2429.
  
\bibitem{SASD} A. J. Scholl, \emph{ Modular forms and de Rham cohomology; Atkin-Swinnerton-Dyer congruences}
 Invent. Math. \textbf{79}  (1985), 49--77.

\bibitem{SMF} A. J. Scholl,  \emph{Motives for modular forms}  Invent. Math. \textbf{100} (1990), no. 2, 419-–430. 

\bibitem{Roh} D. Rohrlich, \emph{Points at infinity on the Fermat
curves} Invent. Math. \textbf{39} (1977), 95--127.

\bibitem{Shi} G.Shimura, \emph{Introduction to the arithmetic theory of automorphic functions}, Kano Memorial Lectures, no. 1. Publications of the Mathematical Society of Japan, no. 11.

\bibitem{TY} T. Yang, \emph{ Cusp form of weight 1 associated to Fermat curves}
 Duke Math J. \textbf{83}  (1996), 141--156.

\bibitem{Weil} A. Weil, \emph{Jacobi sums as ``Gr\"ossencharaktere''} Trans. Amer. Math. Soc. \textbf{73}, (1952). 487–-495.

\end{thebibliography}

\end{document}